\documentclass[11pt]{amsart}
\usepackage{amsmath,amssymb,amscd}
\usepackage{amsmath,amssymb}

\usepackage{color}
\usepackage{amsxtra, amsmath}
\usepackage{amssymb, amscd}
\usepackage{graphicx}

\newcommand\NN{\mathbb N}

\newcommand\tr{\operatorname{tr}}

\theoremstyle{plain}
\newtheorem{thm}{Theorem}[section]

\newtheorem{prop}[thm]{Proposition}

\theoremstyle{definition}

\theoremstyle{remark}
\newtheorem{remark}[thm]{Remark}

\title[ A further look at time-and-band limiting for matrix orthogonal polynomials ]{
A further look at time-and-band limiting for matrix orthogonal polynomials}

\author{M. Castro \and  F. A.  Gr\"unbaum \and  I. Pacharoni \and I. Zurri\'an}

\address{Universidad de Sevilla, Departamento de Matem\'atica Aplicada II, EPS c/ Virgen de Africa 7, 41011, Sevilla, Spain}
\email{mirta@us.es}
\thanks{The work of the first author was  partially supported by  MTM2015-6588-C4-1-P (Ministerio de Econom\'ia y Competitividad),
FQM-262, FQM-7276 (Junta de Andaluc\'ia), Feder Funds (European
Union) and the program Campus de Excelencia Internacional of the Ministerio de Educaci\'on, Cultura y Deporte. }

\address{Department of Mathematics, University of California, Berkeley
CA 94705}
\email{grunbaum@math.berkeley.edu}

\address{~CIEM-FaMAF, Universidad Nacional de C\'ordoba, C\'ordoba~5000, Argentina}
\email{pacharon@famaf.unc.edu.ar}
\thanks{The third and fourth authors were partially suported by SeCyT-UNC and by CONICET grant PIP
112-200801-01533.}

\address{Facultad de Matem\'aticas, Pontificia Universidad Cat\'olica, Santiago~7820436, Chile}
\email{zurrian@famaf.unc.edu.ar}
\thanks{The work of the fourth author was also supported by  FONDECYT 3160646. }

\subjclass[2010]{33C45, 22E45, 33C47}

\keywords{Time-band limiting, Matrix valued orthogonal polynomials}

\begin{document}

\begin{abstract} We extend to
a situation involving matrix valued orthogonal polynomials a scalar result that originates in work of Claude Shannon 
and a ground-breaking series of papers by  D. Slepian, H. Landau and H. Pollak at Bell Labs in the 1960's. While these papers feature integral and differential operators acting on scalar valued functions, we are dealing here with integral and differential operators acting on matrix valued functions.

\end{abstract}

\maketitle

\section{Introduction}

 Claude Shannon, \cite{S}, posed the question of how to best use
the values of the Fourier transform $\mathcal Ff(k)$ of $f$ for values of $k$ in the band $[-W,W]$ when $f(t)$ is a time limited signal.

A detailed account of how this led to  
the series of papers
by three workers at Bell labs in the 1960's: David Slepian, Henry Landau and Henry Pollak, see \cite{SLP1,SLP2,SLP3,SLP4,SLP5} is given in \cite{CG5,GPZ1,GPZ2} and need not be repeated here. Readers unfamiliar with these Bell Lab papers may want to look at these last references.

\bigskip

With this motivation at hand, we can give an account of what we do in this paper: we start with a (matrix valued) version of
a second order differential operator.  Here Shannon would have started with the (scalar valued) second derivative.

We then build the analog of the ``time-and-band limiting" integral  operator,  which we will denote by $S$. We then show that the same ``lucky accident" that the workers at Bell labs found holds here too: we can exhibit a second order differential operator , denoted by $\widetilde{D}$ such that$$S \widetilde{D}  = \widetilde{D}  S.$$

This has, as in the original case of Shannon, very important numerical consequences: it gives a reliable way to compute the eigenvectors of $S$, something that cannot be done otherwise.

\bigskip

The eigenfunctions of $S$ and $\widetilde {D}$ are the same, but using the differential operator instead of the integral one, we have a manageable numerical problem: while the integral operator has a spectrum with eigenvalues that are extremely close together, the differential one has a very spread out spectrum, resulting in a stable numerical computation.

\bigskip

Previous explorations of the commutativity property above in the matrix valued case can be seen in \cite{GPZ1,CG5}, dealing with a full matrix and a narrow banded one, and in \cite{GPZ2}, dealing with an integral operator. In the present paper we extend the work started in the previous references.

\bigskip

For more details on computational issues see \cite{BK14,JKS,ORX}. For applications involving (sometimes) vector-valued quantities on the sphere, see \cite{JB,PS,SD,SDW}.

\section{Preliminaries}

Let $W=W(x)$ be a   weight matrix of size $R$ in the open interval $(a,b)$. By this we mean
a complex $R \times R$-matrix valued integrable function $W$ on the interval $(a, b)$ such that
$W (x)$ is positive definitive almost everywhere and with finite moments of all orders.
 Let $Q_w(x), w=0,1,2,...$, be a sequence of real valued matrix orthonormal polynomials with respect to the weight $W(x)$.
 Consider the following two Hilbert spaces: the
space $L^2((a,b), W(t)dt)$, denoted here by $L^2(W)$, of all matrix valued measurable matrix valued functions $f(x)$, $x\in (a,b)$, satisfying
$\int_a^b \tr\left(f(x)W(x)f^*(x)\right)dx < \infty $ and the space $\ell^2(M_R, \NN_0)$
of all real valued $R\times R$ matrix sequences $(C_w)_{w\in \NN_0}$ such that $\sum_{w=0}^\infty \tr \left(
C_w\,C_w^*\right) < \infty$.

The map $\mathcal F:\ell^2(M_R,\NN_0) \longrightarrow L^2(W)$ given by
$$(A_w)_{w=0}^\infty \longmapsto  \sum_{w=0}^\infty A_w Q_w(x)$$
is an isometry. If the polynomials are dense in $L^2(W)$, this map is unitary with the inverse $\mathcal F^{-1}: L^2(W)\longrightarrow \ell^2(M_R,\NN_0) $ given by
$$ f \longmapsto A_w=\int_a^b f(x)\,W(x)\, Q^*_w(x) dx.$$

We denote our map by $\mathcal F$ to remind ourselves of the usual Fourier transform. Here $\NN_0$ takes up the role of ``physical space"
and the interval $(a,b)$ the role of ``frequency space". This is, clearly,
a noncommutative extension of the problem raised by C. Shannon since he was concerned with scalar valued functions and we are dealing with matrix valued ones.

The {\em time limiting  operator}, at level $N$, acts on $\ell^2(M_R,\NN_0)$ by simply setting equal to zero all the components with index larger than $N$. We denote it by $\chi_N$. The {\em band limiting operator}, at level $\Omega$,  acts on $L^2(W)$ by multiplication by the characteristic function of the interval $(a, \Omega)$, $\Omega\le b$. This operator will be denoted by $\chi_\Omega$.
One could consider restricting the band to an arbitrary subinterval $(a_1,b_1)$. However, the algebraic properties exhibited here, see Section 4  and beyond, hold only with
 this restriction. A similar situation arises in the classical case going all the way back to Shannon.

\medskip
Consider the problem of determining a function $f$,  from the following data: $f$ has support on the finite set $\{0,\dots , N\}$ and its Fourier transform $\mathcal Ff$ is known on the compact set $[a,\Omega]$. This can be formalized as follows
$$\chi_\Omega \mathcal Ff=g=\text{ known },\qquad  \chi_N f=f.$$
We can combine the two equations into
$$E f= \chi_\Omega \mathcal F \chi_N f=g.$$
To analyze this problem we need to compute the singular vectors (and values) of the operator $E:\ell^2(M_R,\NN_0)\longrightarrow L^2(W) $.
These are given by the eigenvectors of the operators
$$E^*E= \chi_N \mathcal F^{-1} \chi_\Omega \mathcal F \chi_N\qquad \text{ and } \qquad S_2=E E^*= \chi_\Omega \mathcal F \chi_N \mathcal F^{-1} \chi_\Omega.$$
The operator $E^*E$, acting in $\ell^2(M_R,\NN_0)$ is just a finite dimensional block-matrix $M$, and each block is given by
$$(M)_{m,n}=(E^*E)_{m,n}= \int_a^\Omega Q_m(x) W(x) Q{^*}_n(x)  dx, \qquad 0\leq m,n \leq N.$$
The second operator $S= E E^*$ acts in $L^2((a,\Omega), W(t)dt)$ by means of the integral kernel
$$k(x,y)=\sum_{n=0}^N Q_n^*(x)Q_n(y).$$
Consider now the problem of finding the eigenfunctions of $E^*E$ and $E E^*$. For arbitrary $N$ and $\Omega$ there is no hope of doing this analytically, and one has to resort to numerical methods and this is not an easy problem. Of all the strategies one can dream of for solving this problem, none sounds so appealing as that of finding an operator with simple spectrum which would have the same eigenfunctions as the original operators. This is exactly what Slepian, Landau and Pollak did in the scalar case, when dealing with the real line and the actual Fourier transform. They discovered (the analog of) the following properties:
\begin{itemize}
  \item For each $N$, $\Omega$ there exists a symmetric tridiagonal matrix $L$, with simple spectrum, commuting with $M$.
  \item For each $N$, $\Omega$ there exists a selfadjoint differential operator $D$, with simple spectrum, commuting with the integral operator $S=EE^*$.
\end{itemize}

\bigskip

To this day nobody has a simple explanation for these miracles, and this paper displays more instances where this holds. Indeed, there has been a systematic effort
to see if the ``bispectral property" first considered in \cite{DG}, guarantees the commutativity of  these
two operators, a global and a local one. A few papers where this question has been taken up, include \cite{G3,G4,G5,G6,GLP,GY,P1,P2}.

\bigskip

We recall that while  \cite{CG5,GPZ1} deal with the full matrix $E^*E$ alluded to above, here, as well as in \cite{GPZ2}, we deal with the integral operator $S=EE^*$ mentioned above.

\section{An example of matrix valued orthogonal polynomials}

For $\alpha,\beta > -1$, the scalar Jacobi weight is given by
\begin{equation}\label{jacobiesc}w_{\alpha,\beta}(x)=(1-x)^{\alpha}(1+x)^{\beta},\end{equation} supported in the interval $[-1,1]$.
In this paper we consider a Jacobi type weight matrix of dimension two
\begin{equation}\label{peso}
W(x)=W_{(\alpha,\beta)}=\frac{1}{2}\left(\begin{array}{cc}w_{\alpha,\beta}+w_{\beta,\alpha}&-w_{\alpha,\beta}+w_{\beta,\alpha}\\ -w_{\alpha,\beta}+w_{\beta,\alpha}&
w_{\alpha,\beta}+w_{\beta,\alpha} \end{array} \right), \quad x\in [-1,1].
 \end{equation}

\

Let us observe that a particular case of these weight matrices have been studied in \cite{GPZ1, GPZ2} and \cite{CG5}.
In fact, the  weight matrix considered in  \cite{PZ, GPZ1, GPZ2} is
\begin{equation}\label{peso-x}
  W_{p,m}(x)= (1-x^2)^{\tfrac m 2 -1} \begin{pmatrix}
  p\,x^2+m-p & -m x\\ -m x & (m-p)x^2+p
\end{pmatrix},\qquad x\in [-1,1].
\end{equation}

\noindent Taking $\alpha=p-1$, $\beta=p+1$ in \eqref{peso}, we obtain a multiple of the previous weight for the special case $m= 2p$.

\medskip

On the other hand, taking $\beta=\alpha-1$ in \eqref{peso} we  obtain a linear translation  $x\longrightarrow 1+x$ of the weight considered in \cite{CG5},
\begin{equation}\label{pesoant}
W_{\lambda}(x)=[x(2-x)]^{\lambda-3/2}\left(\begin{array}{cc}1&x-1\\x-1&1 \end{array} \right), \quad x\in [0,2],
 \end{equation}
with $\lambda=\alpha+\displaystyle \frac{1}{2}$.


\medskip

A sequence of matrix orthogonal polynomials with respect to the matrix valued inner product going with (\ref{peso}), is given by
\begin{equation} \label{Pwdef}
P_n(x)=P_n^{(\alpha,\beta)}(x)=\frac12\begin{pmatrix}
p_n^{(\alpha,\beta)}(x)+p_n^{(\beta,\alpha)}(x)&-p_n^{(\alpha,\beta)}(x)+p_n^{(\beta,\alpha)}(x)
\\
-p_n^{(\alpha,\beta)}(x)+p_n^{(\beta,\alpha)}(x) &p_n^{(\alpha,\beta)}(x)+p_n^{(\beta,\alpha)}(x)
\end{pmatrix},
\end{equation}
 where $p_n^{(\alpha,\beta)}$ are  the classical Jacobi polynomials %
 $$p_n^{(\alpha,\beta)}= \frac{(\alpha+1)_n}{n!}\mbox{}_2\!F_1 \left(\begin{matrix}-n,\,n+\alpha+\beta\\ \alpha+1\end{matrix};\frac{1-x}{2}\right),  $$
which are orthogonal with respect to the weight $w_{\alpha,\beta}$ 
(see for instance \cite[Chapter VI]{Sz}).

\medskip
The norm of the polynomials $P_n(x)$ is given by
$$
 \langle P_n, P_n \rangle= ||P_n(x)||^2= \int_{-1}^1 P_n(x) W(x) P_n(x)^* \,dx =h_n\mathrm{Id},
$$
%
where $\mathrm{Id}$ is the matrix identity of  size $2\times 2$ and
 \begin{equation} \label{norma}
    h_n=\frac{2^{\alpha+\beta+1}\,\Gamma(\alpha+n+1)\,\Gamma(\beta+n+1)}{(\alpha+\beta+2n+1) \,n!\;\Gamma(\alpha+\beta+n+1)\, }.
 \end{equation}

\


The sequence of orthogonal polynomials $(P_n)_{n}$ 
satisfies the three term recurrence relation
$$
xP_n(x)=A_{n}P_{n+1}(x)+B_nP_n(x)+C_nP_{n-1}(x),\quad n\geq 1,
$$
where \begin{align*}
A_n=&\frac{2(n+1)(n+\alpha+\beta+1)}{(2n+\alpha+\beta+1)(2n+\alpha+\beta+2)}\mathrm{Id},&
B_n=&\frac{\alpha^2-\beta^2}{(2n+\alpha+\beta)(2n+\alpha+\beta+2)}T,\\
C_n=&\frac{2(\alpha+n)(\beta+n)}{(2n+\alpha+\beta)(2n+\alpha+\beta+1)}\mathrm{Id},&
\end{align*}
with
\begin{equation} \label{ss} T=\left(\begin{array}{cc}0&1\\1&0 \end{array}\right).\end{equation}

The sequence of matrix orthogonal polynomials $(P_n)_n$ 
satisfies the following differentiation formula
\begin{equation}\label{difform}
 (1-x^2)\frac{d}{dx}P_n(x)= -nxP_n(x) -\frac{n(\alpha-\beta)}{\alpha+\beta+2n}
 T\, P_n(x)
 +\gamma_{n-1} P_{n-1}(x),
\end{equation}
where
\begin{equation}\label{gama}
\gamma_{n-1}= \frac{2(n+\alpha)(n+\beta)}{\alpha+\beta+2n}.
\end{equation}

\medskip

We  also have  the following Christoffel-Darboux formula for the sequence of orthogonal polynomials $(P_n(x))_n$, introduced for a general sequence of matrix orthogonal polynomials in \cite{D96}.

\begin{equation}\label{cd}
\frac{\kappa_{n-1}}{\kappa_{n}\,h_{n-1}}\left(P_{n-1}^*(y)P_{n}(x)-
   P_{n}^*(y)P_{n-1}(x)\right)=(x-y)
   \sum_{k=0}^{n-1}\frac{P_{k}^*(y)P_{k}(x)}{h_k},
\end{equation}
with
$$\kappa_n=\frac{\Gamma(\alpha+\beta+2n+1)}{2^n\, n! \,\Gamma(\alpha+\beta+n+1)}.$$
Observe that $\kappa_n \textrm{Id}$ is the leading coefficient of the matrix polynomial $P_n(x)$ and we also have 
$$\frac{\kappa_{n-1}}{\kappa_n}=\frac{2n (n+\alpha+\beta)}{(2n+\alpha+\beta)(2n+\alpha+\beta-1)}.$$

\medskip
The matrix polynomial  $P_n$, for each $n \geq 0$, is an eigenfunction of the second order differential operator
\begin{equation} \label{oporiginal}
D=\frac{d^2}{dx^2}(1-x^2)+\frac{d}{dx}\left(-x(\alpha+\beta+2) + (\alpha-\beta)T \right),
\end{equation}
with scalar eigenvalues $\Lambda_n=-n(n+\alpha+\beta+1)$.
We have that the differential operator $D$ can be factorized as
$$ D=  \frac{d}{dx}\left(\frac{d}{dx} \, (1-x^2)W(x)\right)W(x)^{-1},$$
 and therefore the sequence of matrix orthogonal polynomials $(P_n(x))_n$ satisfies
\begin{equation}\label{secord}
 \frac{d}{dx}\left(\frac{dP_n}{dx}(x) \, (1-x^2)W(x)\right)W(x)^{-1}= \Lambda_nP_n(x).
\end{equation}

\medskip


\

%
\section{Time and band limiting. Integral and differential operators}\label{Definition Doo}

Given a sequence of matrix orthonormal  polynomials $\{Q_w\}_{w\geq 0}$ with respect to the weight $W$, we fix a natural number $N$ and $\Omega \in (-1,1)$ and
 we consider the integral kernel
\begin{equation}\label{kernejm}
  k(x,y)=\sum_{w=0}^N Q_w^*(x)Q_w(y).
\end{equation}

\noindent  It defines the integral operator  $S$ acting on $L^2((-1,\Omega), W)$ ``from the right-hand side'':
\begin{equation}\label{Iop}
   (fS)(x)=\int_{-1}^\Omega f(y)W(y)\big(k(x,y)\big)^*dy.
\end{equation}

The restriction to the interval $[1, \Omega]$ implements ``band-limiting'' while the restriction to the
range $0, 1, \ldots,N$ takes care of ``time-limiting''. In the language of \cite{GLP}, where the authors were dealing
with scalar valued functions defined on spheres, the first restriction gives a ``spherical cap'' while
the second one amounts to truncating the expansion in spherical harmonics.

\

We search for a selfadjoint differential operator $\widetilde{D}$, defined
in  $[-1, \Omega]$, commuting with the integral operator $S$.
\medskip

The main result of this section is the following

\begin{thm}\label{conmutador}
Let $p(x)=(1-x^2)W(x)$.
The symmetric second-order differential operator
\begin{equation}\label{opconmuta}
\widetilde{D}=\frac{d}{dx}\left((x-\Omega)\frac{d}{dx}p(x)\right)W(x)^{-1}+ x \,A,
\end{equation}
with $A=N(N+\alpha+\beta+2)\mathrm{ Id}$,
  commutes with 
  the integral operator $S$  given  in \eqref{Iop}. 
  \end{thm}

\

 Explicitly, we have
\begin{equation}\label{opD}
\widetilde{D}=\frac{d^2}{dx^2}\widetilde{E}_2(x)+\frac{d}{dx}\widetilde{E}_1(x)+\widetilde{E}_0(x)
\end{equation}
where the coefficients $\widetilde{E}_j,\ j=0,1,2$, are given by
\begin{align*}
\widetilde{E}_2&=(x-\Omega)(1-x^2)\mathrm{Id},\\
\widetilde{E}_1&=\Big(-(3+\alpha+\beta) x^2+x\,\Omega(2+\alpha+\beta)+1\Big)\mathrm{Id}+(\alpha-\beta)(x-\Omega)T,\\
\widetilde{E}_0& =x\, N(N+\alpha+\beta+2)\mathrm{Id},
\end{align*}
and $T$ is the permutation matrix given in \eqref{ss}.

Let us observe that the differential operator $\widetilde D$ is somehow related to the differential operator $D$, given in \eqref{oporiginal} and \eqref{secord}. Explicitly, we have
 $$\widetilde D= (x-\Omega) D+ (1-x^2)\frac d{dx}+ x A.$$

\medskip


We first show that the operator in (\ref{opconmuta}) is indeed symmetric with respect to the (matrix valued) inner product defined by (\ref{peso}).

\medskip


\begin{prop}
 The differential operator $\widetilde D$ is  a  symmetric operator with respect to
\begin{equation}\label{W-trunc}
  \langle f,g\rangle_\Omega=\int_{-1}^\Omega f(x)W(x)g^*(x)\, dx.
\end{equation}
\end{prop}

\begin{proof}
For an appropriate dense set of functions $f,g$, we have
\begin{align*}
  \langle f\widetilde D, g\rangle_\Omega & = \int_{-1}^\Omega \frac{d}{dx} \left( (x-\Omega)\frac{d f}{dx}(x)\, p(x)\right)g^*(x)\, dx+  \int_{-1}^\Omega xf(x)AW(x)g^*(x) \,dx \displaybreak[0]\\
  & = - \int_{-1}^\Omega\frac{d f}{dx}(x)\,  (x-\Omega) p(x)\frac{dg^*}{dx}(x)\, dx+ { (x-\Omega)\frac{d f}{dx}(x)\, p(x)g^*(x)\vline}^{\, \Omega}_{-1}\\
  &\quad + \int_{-1}^\Omega xf(x)AW(x)g^*(x) \,dx \displaybreak[0]\\
  & =  \int_{-1}^\Omega f(x)\,\frac{d }{dx} \Big ( (x-\Omega) p(x)\frac{dg^*}{dx}(x)\Big)\, dx - {f(x)(x-\Omega)p(x)\frac{dg^*}{dx}(x)\vline}^{\, \Omega}_{-1}  \\
  & \quad + { (x-\Omega)\frac{d f}{dx}(x)\, p(x)\frac{dg^*}{dx}(x)\vline}^{\, \Omega}_{-1}+ \int_{-1}^\Omega xf(x)AW(x)g^*(x) \,dx .\\
\end{align*}
Since the factor $(x-\Omega)p(x)$ vanishes at $x=-1$ and $x=\Omega$ we get
\begin{align*}
  \langle f\widetilde D, g\rangle_\Omega & = \int_{-1}^\Omega f(x)W(x)\, \Big (\frac{d }{dx}  \left[(x-\Omega)\frac{d g}{dx}(x)p(x)\right]W^{-1}\Big)^*(x)\, dx
  + \int_{-1}^\Omega xf(x)W(x)\left(g(x)A\right)^* \,dx
\end{align*}
Therefore
$$  \langle f\widetilde D, g\rangle_\Omega = \langle f, g\widetilde D \rangle_\Omega .$$
Completing the proof of the proposition.
\end{proof}


\begin{proof}[{\it  Proof of  Theorem \ref{conmutador}}]
From \cite[Proposition 3.1]{GPZ2} we have that
a symmetric differential operator $\widetilde D$ commutes with  an  integral operator $S$ with kernel $k$ if and only if
 $$  \left( k(x,y)^*\right)\widetilde D_x= (k(x,y)\widetilde D_y)^*.$$
 (Here we use $D_x$ to stress that $D$ acts on the variable $x$).

 Let $D$ be the differential operator introduced in \eqref{oporiginal} and let $(Q_n)_n$ be the orthonormal sequence of matrix valued polynomials given by
  $$Q_n(x)=h_n^{-1/2} \, P_n(x),$$
where $h_n=\langle P_n, P_n\rangle$ is explicitly  given in \eqref{norma}.

These polynomials are eigenfunctions of the differential operator $D$, i.e. $Q_n D=\Lambda_n Q_n$, with $\Lambda_n=-n(n+\alpha+\beta+1)$ and they satisfy
\begin{equation*}
 (1-x^2)\frac{d}{dx}Q_n(x)= -nxQ_n(x) -\frac{n(\alpha-\beta)}{\alpha+\beta+2n}
 S\, Q_n(x) +\tilde \gamma_{n-1} Q_{n-1}(x).
\end{equation*}
with
$$\tilde \gamma_{n-1}=\frac{2(n+\alpha)(n+\beta)}{\alpha+\beta+2n} \frac{h_{n-1}^{1/2}}{h_{n}^{1/2}}.$$
Then
\begin{align*}
  (k(x,y)\widetilde D_y )^*
  &=(y-\Omega)\sum_{n=0}^N  Q_n^*(y)\Lambda_n Q_n(x)+(1-y^2) \sum_{n=0}^N  \frac{d}{dy}Q_n^*(y) Q_n(x)
  +y\, \sum_{n=0}^N  A Q_n^*(y)Q_n(x) \displaybreak[0]\\
  & = (y-\Omega)\sum_{n=0}^N \Lambda_n Q_n^*(y) Q_n(x)+ y \sum_{n=0}^N A Q_n^*(y)Q_n(x)
  \\ &
 \qquad + \sum_{n=0}^N \left( -y  n Q_n^*(y)Q_n(x) -  \tfrac{n(\alpha-\beta)}{(\alpha+\beta+2n)} Q_n^*(y) S Q_n(x) + \gamma_{n-1} Q_{n-1}^*(y)Q_n(x) \right)
\end{align*}
and similarly
\begin{align*}
 \left( k(x,y)^*\right)\widetilde D_x
  &= 
   (x-\Omega)\sum_{n=0}^N \Lambda_n Q_n^*(y) Q_n(x)+ x \sum_{n=0}^N A Q_n^*(y)Q_n(x)
  \\ &
 \qquad + \sum_{n=0}^N \left( -x  n Q_n^*(y)Q_n(x) -\displaystyle  \tfrac{n(\alpha-\beta)}{(\alpha+\beta+2n)} Q_n^*(y) S Q_n(x) +\tilde \gamma_{n-1} Q_{n-1}^*(y)Q_n(x) \right).
\end{align*}
Thus
\begin{equation}
\begin{aligned}\label{a}
( k^*\widetilde D_x)- ( k\widetilde D_y)^*&= (x-y) \sum_{n=0}^N (\Lambda_n+A-n) Q_n^*(y)Q_n(x) \\ & \qquad +  \sum_{n=0}^N  \tilde \gamma_{n-1} \Big(Q_n^*(y)Q_{n-1}(x)-Q_{n-1}^*(y)Q_n(x)  \Big).
\end{aligned}
\end{equation}

From the Cristoffel-Darboux formula, given in \eqref{cd}, we obtain that
\begin{equation*}
\tilde \gamma_{n-1} \Big( Q_{n-1}^*(y)Q_{n}(x)-
   Q_{n}^*(y)Q_{n-1}(x)\Big)
=(x-y)\,\tilde \gamma_{n-1} \frac{\kappa_{n}\,h_{n-1}^{1/2}} {\kappa_{n-1} h_n^{1/2}}\,
   \sum_{k=0}^{n-1}Q_{k}^*(y)Q_{k}(x).
\end{equation*}
It is easy to verify that
$\displaystyle \tilde \gamma_{n-1} \frac{\kappa_{n}\,h_{n-1}^{1/2}} {\kappa_{n-1} h_n^{1/2}}= \alpha+\beta+2n+1.$
Therefore, by exchanging the order of summation we get
\begin{align*}
\sum_{n=0}^N \tilde \gamma_{n-1} \Big(&  Q_{n-1}^*(y)Q_{n}(x) -
   Q_{n}^*(y)Q_{n-1}(x)\Big)
=(x-y) \sum_{n=0}^N  (\alpha+\beta+2n+1)
   \sum_{k=0}^{n-1}Q_{k}^*(y)Q_{k}(x) \displaybreak[0]\\
&=   (x-y) \sum_{n=0}^{N-1} \sum_{j=n+1}^{N}  (\alpha+\beta+2j+1)    Q_{n}^*(y)Q_{n}(x) \displaybreak[0]\\
&=   (x-y) \sum_{n=0}^{N-1} \big(  N(N+\alpha+\beta+2)-n(n+\alpha+\beta+2)\big)    Q_{n}^*(y)Q_{n}(x).
\end{align*}

Now from \eqref{a},  and using that $\Lambda_n+A-n= -n(n+\alpha+\beta+2)+N(N+\alpha+\beta+2)$ we  get
$$ (k(x,y)\widetilde D_y )^* =  \left( k(x,y)^*\right)\widetilde D_x. $$
Hence the operators $\widetilde D$ and $S$ commute.
\end{proof}

\medskip

 \begin{remark}
    It is worth noticing that if one takes  the new operator $\widetilde{D}^{(1)}=\widetilde{D}T$, one obtains an  operator commuting with the integral operator $S$ in (\ref{Iop}), linearly independent with the operator $\widetilde{D}$. The space of  the differential operators of order two commuting with $S$ is generated by $T$, $\widetilde{D}T$ and $\widetilde{D}$.
  \end{remark}

\section{A Chebyshev type example}

As a particular example, if we put $\alpha=\displaystyle \frac{1}{2}$, $\beta=-\displaystyle \frac{1}{2}$  in (\ref{peso}), we have the Chebyshev type weight
\begin{equation}\label{pesoCheb}
W_{(\frac{1}{2},-\frac{1}{2})}(x)=\frac{1}{\sqrt{1-x^2}}\left(\begin{array}{cc}1&x\\x&1 \end{array} \right), \quad x\in [-1,1],
 \end{equation}
 which was introduced in  \cite[Page 586]{Be}
 for a different purpose.

 The monic family of polynomials orthogonal with respect to this weight matrix, is given explicitly in terms of the Chebyshev polynomials of the second kind $U_n(x)$,
 $$ \widetilde{P}_n(x)=\frac{1}{2^n}\left( \begin{array}{cc}U_n(x)&-U_{n-1}(x)\\-U_{n-1}(x)&U_n(x)  \end{array}\right), $$
and it satisfies a first order differential equation as pointed out in \cite{CG1,CG2}.

Moreover, the polynomials $P_n(x)=P^{(\alpha,\alpha-1)}_n(x)$ in \eqref{Pwdef}
satisfy, for any $\alpha>0$, the first order differential equation
$$
P_n'(x)\begin{pmatrix}-x&1\\-1&-x \end{pmatrix}
+P_n(x)\begin{pmatrix} -2\alpha&0\\0&0 \end{pmatrix}=\begin{pmatrix}-2\alpha-n&0\\0&n\end{pmatrix}P_n(x),
$$ as shown in \cite{CG5}.

One considers here the integral operator $S$ in \eqref{Iop} defined by the integral kernel
$$
k(x,y)=\sum_{n=0}^N4^n\widetilde{P}_n(x)^*\widetilde{P}_n(y).
$$
Particularly,  the norm of the polynomials $\widetilde{P}_n$ is given by $\displaystyle
||\widetilde{P}_n||=\frac{\sqrt{\pi}}{2^n}$.

Hence, for this particular example, the commuting  operator $\widetilde{D}$ is given by
$$
\widetilde{D}=\frac{d^2}{dx^2}(1-x^2)(x-\Omega)+\frac{d}{dx}\Big((-3x^2+2\Omega\, x+1)\mathrm{Id}+(x-\Omega)T\Big)+N(N+2)x,
$$
where $T$ is the permutation matrix in \eqref{ss}.

 \medskip

\section{Conclusion and outlook}
The main result derived in the previous sections is the existence of an explicit differential operator $\widetilde{D}$, which, as we proved, commutes with $S$.

    If one compares this result with the one in the celebrated series of papers by D.Slepian, H.Landau
and H.Pollak one may say that we are at the stage of their first paper. What is needed now is
an argument to conclude that the eigenfunctions of $\widetilde{D}$ will automatically be eigenfunctions of the
integral operator $S$.

    In the series of papers mentioned above the simplicity of the spectrum of $\widetilde{D}$ follows from classical
Sturm-Liouville theory and this guarantees that they have found a good way to compute the
eigenvectors of $S$.
In our situation, things could eventually be reduced to that case, but in principle $\widetilde{D}$, as well as $S$,
have ``matrix valued eigenvalues'', and the appropriate notion of ``simple spectrum'' requires careful
handling.

\medskip

To be quite explicit: there are two technical points that we are not addressing in this paper. The first one is the issue of ``simple spectrum'' in the
matrix valued context. A good place to look for this sort of question is \cite {RBK}. A second issue is that of making precise
the correct selfadjoint extension ( i.e. boundary conditions) for our second order differential operator $\widetilde{D}$. This  issue has been
implicit starting with the first papers  of Slepian, Landau and Pollak, and has been considered quite explicitly
in the very recent paper \cite{KV}.

\end{document}